\documentclass[a4paper]{jpconf}
\usepackage{amsfonts,amsthm,amsmath,amssymb,
	amsbsy,euscript} 

\newtheorem{theor}{Theorem}
\theoremstyle{definition}
\newtheorem*{convention}{Convention}
\newtheorem{proposition}[theor]{Proposition}
\newtheorem{define}{Definition}
\newtheorem{problem}{Problem}
\newtheorem{open}[problem]{Open problem}
\theoremstyle{remark}
\newtheorem{rem}{Remark}
\newtheorem*{notation}{Notation}

\newcommand{\BBR}{\mathbb{R}}

\newcommand{\dd}{\partial}
\newcommand{\schouten}[1]{\lshad {#1} \rshad}

\DeclareMathOperator{\Assoc}{Assoc}

\newcommand{\by}[1]{\textup{{#1}}}
\newcommand{\jour}[1]{\textit{{#1}}}
\newcommand{\vol}[1]{\textbf{{#1}}}

\begin{document}
\title{Associativity certificates for Kontsevich's star\/-\/product
$\star$~mod~$\bar{o}(\hbar^k)$: $k\leqslant6$ unlike~$k\geqslant7$}

\author{Ricardo Buring${}^{*,}$\footnote[1]{Present address: Centre INRIA de Saclay \^Ile\/-\/de\/-\/France, B\^at.~Alan Turing, 1~rue Honor\'e\ d'Estienne d'Orves, 
F-91120 Palaiseau, France} 
   and Arthemy V Kiselev${}^\S$}

\address{${}^{*}$ Institut f\"ur Mathematik, 
Johannes Gutenberg\/--\/Uni\-ver\-si\-t\"at,
Staudingerweg~9, 
\mbox{D-\/55128} Mainz, Germany}

\address{${}^{\S}$ Ber\-nou\-lli Institute for Mathematics, Computer Science and Artificial Intelligence, University of Groningen, P.O.~Box 407, 9700~AK Groningen, The Netherlands}

\ead{A.V.Kiselev@rug.nl}

\begin{abstract}
The formula $\star$ mod~$\bar{o}(\hbar^k)$ of Kontsevich's star\/-\/product with harmonic propagators was known in full at~$\hbar^{k\leqslant6}$ since 2018 for generic Poisson brackets, and since 2022 also at~$k=7$ for affine brackets.
We discover that the mechanism of associativity for the star\/-\/product up to $\bar{o}(\hbar^6)$ is different from the mechanism at order~$7$
for both the full star\/-\/product and the affine star\/-\/product.
Namely, at lower orders the needed consequences of the Jacobi identity are immediately obtained from the associator mod~$\bar{o}(\hbar^6)$, whereas at order~$\hbar^7$ and higher, some of the necessary differential consequences are reached from the Kontsevich graphs in the associator in strictly more than one step.
\end{abstract}

\noindent\textbf{Introduction.}\quad
Deformation quantization extends the commutative associative unital product~$\times$ in the algebra $A\mathrel{{:}{=}} C^\infty(M)$ of smooth functions on a manifold~$M$ to an associative product~$\star$ on the space of formal power series~$A[[\hbar]]$; the skew\/-\/symmetric part of the bi\/-\/derivation in the leading deformation term at~$\hbar^1$ in~$\star$ is readily seen to be a Poisson bracket $\{{\cdot},{\cdot}\}_P$ on the algebra~$A$.

\begin{theor}[\cite{MK97}]
For every Poisson bi\/-\/vector $P$ on a finite\/-\/dimensional affine real manifold $M$ and an infinitesimal deformation 
$\times \mapsto \times + \hbar\,\{{\cdot},{\cdot}\}_P + \bar{o}(\hbar)$ 
towards the respective Poisson bracket, 
there exists a system of weights $w(\Gamma)$, 
uniformly given by an integral formula, 
such that the $\BBR[[\hbar]]$-\/bilinear 
star\/-\/product,
\begin{equation}\label{EqStarMK}
\star = \times + \sum
_{n\geqslant1} \frac{\hbar^n}{n!} \sum
_{\Gamma\in\hat{G}^n_2}  w(\Gamma) \cdot \Gamma(P,\ldots,P )(\cdot,\cdot),
\end{equation}
is associative\textup{;}
here $\hat{G}^n_m \subset G^n_m$ is the subset of Kontsevich graphs built of wedges \textup{(}with each aerial vertex having exactly two outgoing edges\textup{)} 
in the set~$G^n_m$ of all Formality graphs with $m$~ground vertices and $n$~aerial vertices.\footnote{\label{FootDefFormalityGraph}%
By definition, a \emph{Formality graph} is a simple directed graph (that is, without double edges and without tadpoles) on $m + n$ vertices $\{ 0$, $\ldots$, $m-1$, $m$, $\ldots$, $m+n-1\}$, such that the $m$ ground vertices $0$, $\ldots$, $m-1$ are sinks (with no outgoing edges) and the $n$ vertices $m$, $\ldots$, $m+n-1$ are called \emph{aerial}. 
The set of outgoing edges at each vertex 
is endowed with a total ordering: Left${}\prec{}$Right, 
Left${}\prec{}$Middle${}\prec{}$Right,~etc.}
\end{theor}

\begin{convention}
To the edges $L$ and $R$ of the wedge graph $\Lambda$ we ascribe independent indices $i$ and $j$ respectively, and with this graph $\Lambda$ we associate the operator $\Lambda(P)(f,g)=P^{ij}\cdot\dd_i f\cdot\dd_j g$ which is the Poisson bracket. 
More generally, for a Formality graph $\Gamma \in G^n_m$ we ascribe independent indices to all edges; the multi\/-\/linear multi\/-\/differential operator $\Gamma(J_0$, $\ldots$, $J_{n-1})(f_0$, $\ldots$, $f_{m-1})$ associated with the graph $\Gamma$ is then a sum over those indices, with each summand being a product over the (differentiated)
contents of vertices, the ground vertex $k$ containing the argument $f_k$ of the operator, and the aerial vertex $m + \ell$ containing the component of the multi\/-\/vector field $J_\ell$ specified
by indices of the ordered outgoing edges; here the content of each vertex is differentiated
with respect to the local affine coordinates specified by the incoming edges (if~any).
\end{convention}

Elementary properties of the graph weights $w(\Gamma)$
are summarized in~\cite[Lemmas~1--5 and Remark~8]{cpp}; 
the Shoikhet\/--\/Felder\/--\/Willwacher cyclic weight relations from~\cite[App.~E]{FelderWillwacher} are recalled in~\cite[Proposition~7]{cpp}. (These relations are not enough to determine the weights completely.)
Another ample source of relations between weights is 
the associativity of~$\star$;
this can be exploited as in~\cite[Methods~1--3]{cpp}. 

The Kontsevich star\/-\/product with harmonic propagators (as in~\cite{MK97}) was known at orders $\hbar^1$, $\ldots$, $\hbar^4$ in~2017 from~\cite{cpp}. The weights $w(\Gamma)$ of all Kontsevich graphs at $\hbar^5$ and~$\hbar^6$ in~$\star$ were obtained by the end of~2018 in~\cite{BPP}; the Riemann zeta value $\zeta(3)^2 / \pi^6$ starts appearing in the weights $w(\Gamma)$ for some Kontsevich graphs $\Gamma \in \smash{\hat{G}^n_2}$ from $n=6$ onwards.
In~\cite{affine7} from~2022 we found the weights $w(\Gamma)$ of all Kontsevich graphs $\Gamma \in \smash{\hat{G}^{n=7}_{m=2}}$ with in-degree${}\leqslant 1$ of aerial vertices, that is the weights of the graphs which are relevant for the case of \emph{affine} Poisson brackets $\{\cdot,\cdot\}_P$ and thus, as~$n\leqslant 7$, for the seventh order expansion $\star_{\text{aff}}$ mod~$\bar{o}(\hbar^7)$ of the affine star\/-\/product.
We then establish in~\cite{affine7} that the entire coefficient of~$\zeta(3)^2 / \pi^6$, which does show up in~$\star_{\text{aff}}$ at $\hbar^6$ and~$\hbar^7$, equals a linear combination of differential consequences of the Jacobi identity (for affine Poisson brackets $\{\cdot,\cdot\}_P$) because the respective linear combination of Kontsevich graphs near~$\zeta(3)^2 / \pi^6$  assimilates into a linear combination of \emph{Leibniz} graphs on $m=2$ ground vertices and $\tilde{n}=n-1$ aerial vertices.

\begin{define}
A \emph{Leibniz graph} is a Formality graph containing at least one aerial vertex with three outgoing edges, such that those three edges have three distinct targets, and none of those three edges are tadpoles. The other aerial vertices (if any) have two outgoing edges, and the ground vertices are as usual. These graphs will be evaluated with the \emph{Jacobiator} $\tfrac{1}{2}\lshad P,P \rshad$ of the Poisson structure~$P$ in the vertex with three outgoing edges, hence representing a differential operator that is identically zero whenever $P$~is Poisson.\footnote{%
Homogeneous components of differential consequences of the Jacobi identity (now realized by using Leibniz graphs) vanish separately thanks to the following lemma: 
A tri\/-\/differential operator $\sum_{ |I|,|J|,|K|\geqslant 0} c^{IJK} \dd_I \otimes \dd_J \otimes \dd_K$ vanishes identically iff all its coefficients vanish: $c^{IJK} = 0$ for every triple $(I, J, K)$ of multi\/-\/indices;
here $\dd_L = \dd_1^{\alpha_1} \circ \cdots \circ \dd_n^{\alpha_n}$ for a multi\/-\/index $L = (\alpha_1, \ldots, \alpha_n )$. Moreover, the sums 
$\sum_{|I|=i,|J|=j,|K|=k} c^{IJK} \dd_I \otimes \dd_J \otimes \dd_K$
are then zero for all homogeneity orders~$(i,j,k)$.}
\end{define}

We recall from the breakthrough paper~\cite{MK97} the guaranteed existence of a factorization of the star\/-\/product associator,
$
\Assoc(\star)(P)(f,g,h) = (f\star g)\star h - f\star(g\star h)$
with $f,g,h \in A[[\hbar]]$, 
via Leibniz graphs (here $M$~is an affine real manifold of finite dimension~$d$).

\begin{proposition}[Corollary~4 and Conjecture ending~\S4 in~\cite{kiev19}]\label{PropFactorAssocWithCn}
The operator $\Diamond$ that solves the factorization problem
\begin{equation}\label{EqFactorAssoc}
\Assoc(\star)(P)(f, g, h) = \lozenge \bigl( P, \schouten{P, P}\bigr) (f, g, h),
 \qquad f,g,h\in A[[\hbar]],
\end{equation}
is given by
\begin{equation}\label{EqDiamond}
\lozenge = 2 \cdot \sum\nolimits_{n\geqslant 1} \frac{\hbar^n}{n!} \cdot
 c_n \cdot \mathcal{F}_{n-1} \bigl( \schouten{P, P}, P, \ldots, P\bigr),
\end{equation}
where $\mathcal{F}_k$~is the $k$-\/ary component of the Formality $L_\infty$-\/morphism, and where we claim that the constants~$c_n$ are equal to~$n/6$.
\end{proposition}

The number of graphs which actually show up at order~$\hbar^k$ in the left-{} and right\/-\/hand sides of factorization problem~\eqref{EqFactorAssoc} is reported in Table~\ref{TabMKLeibnizgraphsAssoc}.
\begin{table}[htb]
\caption{\label{TabMKLeibnizgraphsAssoc}%
The number of graphs in either side of the associator's factorization.}
\begin{center}  \begin{tabular}{lrrrrrr}  \br
$k$ &  2 & 3 & 4 & 5 & 6 & 7\\
\mr
LHS: \# Kontsevich graphs, &  3\,(Jac) & 39 & 740 & 12464 & 290305 & ?\\
\mbox{ }\qquad coeff${}\neq 0$ &&&&&& \\
RHS: \# Leibniz graphs, &  1\,(Jac) & 13 & 241 & 4609 & ? &  ?\\
\mbox{ }\qquad coeff${}\neq 0$ &&&&&& \\
\br
\end{tabular}
\end{center}
\end{table}
For example, in~\cite[\S5]{kiev19} we inspect many graphs of different orders, and establish the equality of sums of Kontsevich graphs in the associator and sums of Leibniz graphs --\,in the factorizing operator\,-- after they are expanded into the Kontsevich graphs.

In the right\/-\/hand side of the associator for~$\star$, there are Leibniz graphs: at $\hbar^{k\geqslant2}$, such Leibniz graphs have $3$~sinks, $k-1$ aerial vertices (of which one vertex, the Jacobiator, has three outgoing edges, and the remaining $k-2$ vertices (if any) each have two outgoing
edges), and, by the above, $3+(k-2)\cdot 2 = 2k-1$ edges; tadpoles are not allowed, graphs with multiple edges are discarded. For each $k = 2$,\ $3$,\ $4$,\ $5$ we generate all such admissible Leibniz graphs (those can be \emph{zero} graphs with a parity\/-\/reversing automorphism, cf.~\cite{cpp}); the respective number of such Leibniz graphs at each order~$\hbar^k$ is in Table~\ref{TabNumberLeibnizAssoc}.
\begin{table}[htb]
\caption{\label{TabNumberLeibnizAssoc}%
The count of admissible Leibniz graphs in the associator for Kontsevich's~$\star$.}
\begin{center}  \begin{tabular}{lrrrrr} \br
$k$ & 2 & 3 & 4 & 5 & 6\\
\mr
\# Leibniz graphs, generated & 1 & 24 & 520 & 11680 & 293748\\
\# Leibniz graphs generated, nonzero & 1 & 24 & 490 & 11260 & 285684\\
\# Leibniz graphs generated, nonzero, diff.\,order${}>0$ & 1 & 15 & 301 & 6741 & 171528\\
\# Leibniz graphs (coeff${}\neq0$ in associator) & 1 & 13 & 241 & 4609 & ?\\
\mr
\# Cyclic weight relations & 1 & 15 & 301 & 6741 & 171528\\
Corank of linear algebraic system & 0 & 3 & 66 & 1469 & ?\\
\br \end{tabular}  \end{center}
\end{table}
At every order~$k$, we generate the entire set of the cyclic weight relations (cf.~\cite{FelderWillwacher}); 
every cyclic weight relation is a linear constraint upon the weights of
several Leibniz graphs;
all those weights are given by the Kontsevich integral formula from~\cite{MK97}. 
The number of these linear relations and the (co)\/rank of this linear algebraic system follow in Table~\ref{TabNumberLeibnizAssoc}.

Banks\/--\/Panzer\/--\/Pym in~\cite{BPP} do not list the weights of Leibniz graphs (as in Table~\ref{TabNumberLeibnizAssoc} above), for these graphs do not show up in the $\star$-\/product itself where the vertex\/-\/edge valency is different for the Kontsevich graphs.
We use the software \textsf{kontsevint} by Panzer 
(cf.~\cite{BPP}) 
to calculate the Kontsevich weights of all the Leibniz graphs which are admissible for the right\/-\/hand side of star\/-\/product's associator. 
(Some weights can --\,and actually do\,-- vanish because either the graph is zero, or the weight integrand is identically zero, or the weight formula integrates to a zero number.) 
The count of admissible Leibniz graphs with nonzero weights is in the fourth line of Table~\ref{TabNumberLeibnizAssoc}: the corresponding line in Table~\ref{TabMKLeibnizgraphsAssoc} is reproduced verbatim. 
(The Leibniz graphs with zero weights do nominally show up in the cyclic weight relations for Leibniz graphs, but in fact stay invisible in the formulas.)

\begin{proposition}
The numeric values of the Kontsevich weights~$w(L)$ of Leibniz graphs with~$k$ aerial vertices on $3$~sinks, which we calculated using Panzer's software \textup{\textsf{kontsevint}}, 
do satisfy\footnote{%
The relations are satisfied exactly, without involvement of any conventional 
constants and normalizations (in contrast with the mandatory use of auxiliary constants $c_n = n/6$ in Proposition~\ref{PropFactorAssocWithCn}, see above).
But let us remember that the multiplicativity of Kontsevich weights is more subtle for graphs on three ground vertices than for Kontsevich's graphs on two sinks.} 
the system of linear algebraic equations given by the cyclic weight relations for $k = 1$,$2$,$3$,$4$.
\end{proposition}

From now on in this paper, we study the associativity of Kontsevich's $\star$-\/product from a different perspective, because at~$\hbar^{k\geqslant 6}$ the number of Leibniz graphs --\,to realize the associator of~$\star$ at~$\hbar^k$\,-- is too big for the weight $w(L)$ to be computed for every such graph~$L$ (either before or after the system of cyclic weight relations is formed at~$k\geqslant 6$).

Let us recall that the associator naturally splits into homogeneous orders with respect to the three sinks; 
so does the set of relevant Leibniz graphs. 
(The cyclic weight relations correlate the weights~$w(L)$ of Leibniz graphs for \emph{different} tri\/-\/differential orders in the associator.)
We say that finding the values of Leibniz graph coefficients dictated by the Kontsevich integral formula yields \emph{the} solution of the \emph{strong} factorization problem for the associator of the star\/-\/product.
Yet, to certify the associativity it suffices to find \emph{a} realization of each tri\/-\/differential component as a weighted sum of Leibniz graphs regardless of any such realizations for other tri\/-\/differential orders, that is without imposing the known constraints upon the Leibniz graph weights.
One big problem thus splits into many small subproblems, which are solved independently.
The result is a solution to the \emph{weak} factorization problem, which we report in this paper.
Let us remember that the found values of Leibniz graph coefficients are then not necessarily equal to the Kontsevich integrals~$w(L)$ (times the rational factors which count the multiplicities).

\section*{The layers of Leibniz graphs: contract and expand edges in the Kontsevich graphs}

The idea which we start with is to \emph{not} consider those Leibniz graphs whose expansion --\,of Jacobiators into sums of Kontsevich's graphs, and of all the derivations acting on the Jacobiators by the Leibniz rule\,-- does not reproduce any of the Kontsevich graphs in the associator itself. 

By definition, the 
\emph{$0$th layer} of Leibniz graphs is obtained 
--\,for a given linear combination of Kontsevich graphs\,--
by contracting one internal edge in every Kontsevich graph in all possible ways.
By expanding the $0$th layer Leibniz graphs back to Kontsevich's graphs, one reproduces their original set, but \emph{new} Kontsevich's graphs can be obtained. The coefficients of these new Kontsevich graphs, not initially present in the given linear combination, either cancel out or do not all vanish identically. 
If not, then by repeating for those new Kontsevich graphs the above contraction\/-\/expansion procedure, one reproduces (part of) the $0$th layer but also produces 
the $1$st layer of new Leibniz graphs and from them, possibly a still larger set of Kontsevich's graphs (whose number is finite for a given number of aerial vertices). 
The construction of layers is iterated until saturation
(e.g., see Table~1 in~\cite{JPCS2017}). 
Even if the saturation requires layers to achieve, the resulting number of Leibniz graphs at hand is much smaller than the number of Leibniz graphs on equally many vertices.\footnote{%
The attribution of Leibniz graphs to layers seems to depend on the choice of propagator in the integral formula of~$w(\Gamma)$, cf.~\cite{OperadsMotives1999}: the choice dictates the set of Kontsevich graphs actually showing up in the associator (whence the $0$th layer of Leibniz graphs).}
Moreover, in this paper running the algorithm until saturation is not obligatory; the first layer of Leibniz graphs
is already enough.

\begin{rem}
To avoid repetitions, we use the normal form of Leibniz graphs; it refers to the encoding of directed graphs in \textsf{nauty} 
within \textsf{SageMath} 
(see
~\ref{AppLeibniz} for further discussion).
\end{rem}

\begin{proposition}[see
~\ref{AppStarAssoc6}]
\label{PropAssoc6}
The Kontsevich $\star$-product with the harmonic graph weights, known up to $\bar{o}(\hbar^6)$ from Banks\/--\/Panzer\/--\/Pym \cite{BPP}, is associative modulo $\bar{o}(\hbar^6)$: every tri-differential homogeneous component of the associator admits \emph{some} realization by Leibniz graphs; to find 
such solution, the $0$th layer of Leibniz graphs suffices for each of the tri-differential orders.
\end{proposition}

\begin{proof}[Proof scheme]
The associativity of Kontsevich's $\star$-product up to $\bar{o}(\hbar^4)$, that is, \[\operatorname{Assoc}(\star(P))(f,g,h)\text{ mod }\bar{o}(\hbar^4) = \Diamond(P, [\![P, P]\!])(f,g,h)\text{ mod }\bar{o}(\hbar^4),\] is the core of paper~\cite{cpp}.
Next, in \cite[Part I, \S 3.5.1]{BuringDisser} we provide a realization of the component $\sim \hbar^5$ in the associator $\operatorname{Assoc}(\star)$ mod $\bar{o}(\hbar^5)$ in terms of the Leibniz graphs from the $0$th layer,
that is, by using the Leibniz graphs obtained at once by contracting edges between aerial vertices in the Kontsevich graphs from the associator.

There are $105$ homogeneous tri-differential order components at $\hbar^6$ in the associator $\operatorname{Assoc}(\star)$ mod $\bar{o}(\hbar^6)$. 
We import the harmonic graph weights at $\hbar^5$ and $\hbar^6$ in $\star$ mod $\bar{o}(\hbar^6)$ from the \textsf{kontsevint} repository of E. Panzer (Oxford).
At order $\hbar^6$, the weights of Kontsevich graphs in $\star$ are expressed as $\mathbb{Q}$-linear combinations of $1$ and $\zeta(3)^2/\pi^6$.
In consequence, the coefficients of Kontsevich graphs in the associator at order $\hbar^6$ are also $\mathbb{Q}$-linear combinations of that kind.
Every tri-differential homogeneous component of the associator is thus split into the rational- and $\zeta(3)^2/\pi^6$-slice: either of the slices is a linear combination of Kontsevich's graphs with rational coefficients.
The rational slices are met in all of the $105$ tri-differential orders; we detect that in every such slice the Kontsevich graphs provide the $0$th layer of Leibniz graphs which suffices to realize that sum of Kontsevich graphs as a linear combination of these Leibniz graphs.
The $\zeta(3)^2/\pi^6$-slice is nontrivial in $28$ tri-differential orders of the associator at $\hbar^6$;
here the Formality mechanism works as follows.
For all but $6$ tri-differential orders, the Kontsevich graphs from the linear combination near $\zeta(3)^2/\pi^6$ suffice to provide the set of $0$th layer Leibniz graphs which are enough for a solution of the factorization problem.
The tri-differential orders $\{(1,1,3),(3,1,1),(2,1,2),(1,2,2),(2,2,1),(1,3,1)\}$ are special:
for a solution to appear, the sets of Kontsevich graphs from the rational and $\zeta(3)^2/\pi^6$-slices within that tri-differential order must be merged
and then the union set is enough to provide a factorization of the $\zeta(3)^2/\pi^6$-slice by the $0$th layer of Leibniz graphs.
The corresponding computations are presented in
~\ref{AppStarAssoc6}. 
We conclude that at order $6$ for the full Kontsevich star\/-\/product, Kontsevich's Formality mechanism works as expected.
\end{proof}


The seventh order expansion of the Kontsevich star\/-\/product for arbitrary Poisson brackets is unknown 
(see Table~1 
in~\cite{affine7} for the count of 2,814,225 Kontsevich graphs at~$\hbar^7$).
So far, in \cite{affine7} we have obtained the \emph{affine} star\/-\/product $\star_{\text{aff}}\text{ mod }\bar{o}(\hbar^7)$ under the assumption that the coefficients of the Poisson bracket are affine functions (e.g., linear on the affine base manifold).
For example such are the Kirillov\/--\/Kostant Poisson brackets on the duals~$\mathfrak{g}^*$ of Lie algebras.
We discover that the associativity mechanism for the affine star\/-\/product at $\hbar^7$ 
differs from the mechanism which worked at lower orders of expansion in $\hbar$ for the full star\/-\/product $\star\text{ mod }\bar{o}(\hbar^6)$.
Moreover, the new mechanism of associativity for $\star_{\text{aff}}\text{ mod }\bar{o}(\hbar^7)$ forces a new mechanism of associativity for the full star\/-\/product $\star\text{ mod }\bar{o}(\hbar^7)$ starting at order seven.
The difference is the necessity of Leibniz graph layers beyond the $0$th layer, which itself was enough at lower orders to build a solution of the weak problem for associator's factorization via the Jacobi identity.

\begin{proposition}
\label{PropStarAffineAssoc}
The affine Kontsevich star\/-\/product expansion $\star_{\textup{aff}}$ \textup{mod} $\bar{o}(\hbar^7)$ found in~\cite[Proposition~8]{affine7} 
is associative modulo $\bar{o}(\hbar^7)$.
Namely, \textup{(}every homogeneous tri-differential component of\textup{)} the associator $(f \star_{\textup{aff}} g) \star_{\textup{aff}} h - f \star_{\textup{aff}} (g \star_{\textup{aff}} h)$ \textup{mod} $\bar{o}(\hbar^7)$ is realized as \emph{some} sum of Leibniz graphs.
\end{proposition}

\begin{proof}[Proof scheme]
With not yet specified undetermined coefficients of Kontsevich graphs at $\hbar^7$ in the affine star\/-\/product $\star_{\text{aff}}$ mod $\bar{o}(\hbar^7)$, its associator's part at $\hbar^7$ expands to 203 tri-differential order components.
As soon as the weights of all the new Kontsevich graphs on $n=7$ aerial vertices are fixed (see~\cite{affine7})
, the number of tri-differential orders $(d_0,d_1,d_2)$ actually showing up at $\hbar^7$ in the associator $\mathsf{A}$ for $\star_\text{aff}$ mod $\bar{o}(\hbar^7)$ drops to $161$.
For all but four tri-differential order components $\mathsf{A}_{d_0 d_1 d_2}$ in the associator $\mathsf{A}$, the $0$th layer of Leibniz graphs, which are obtained by contracting\footnote{Note that the Leibniz graphs in the $0$th layer have vertices of in-degree $\leqslant 2$ because they are obtained by the contraction of a single edge in the Kontsevich graphs with vertices of in-degree $\leqslant 1$.} one edge between aerial vertices in the Kontsevich graphs of every such tri-differential component $\mathsf{A}_{d_0 d_1 d_2}$, is enough to provide \emph{a} solution for the factorization problem, $\mathsf{A}_{d_0 d_1 d_2} = \Diamond_{d_0 d_1 d_2}(P, \schouten{P,P})$, expressing that component by using differential consequences of the Jacobi identity (encoded by Leibniz graphs).
We detect that for the tri-differential orders $(d_0,d_1,d_2)$ in the set $\{(3,3,2), (2,3,3), (3,2,3), (2,4,2)\}$,
the Leibniz graphs from the $0$th layer are not enough to reach a solution $\Diamond_{d_0 d_1 d_2}$; still a solution $\Diamond_{d_0 d_1 d_2}$ appears in each of these four exceptional cases after we add the Leibniz graphs from the $1$st layer (i.e.\ those graphs obtained by contraction of edges in the Kontsevich graph expansion of Leibniz graphs from the previous layer
).
(There are $2294$ Kontsevich graphs in $\mathsf{A}_{2,3,3}$, producing $3584$ Leibniz graphs in the respective $0$th layer immediately after the edge contractions;
the component $\mathsf{A}_{3,3,2}$ contains equally many Kontsevich graphs and the same number of Leibniz graphs in the $0$th layer;
the largest component $\mathsf{A}_{3,2,3}$ contains $2331$ Kontsevich graphs and gives $3603$ Leibniz graphs in the $0$th layer;
and finally $\mathsf{A}_{2,4,2}$ contains $1246$ Kontsevich graphs and produces $2041$ Leibniz graphs in the $0$th layer.)
In \cite[Part I, \S 3.7.8]{BuringDisser} we generate \emph{a} Leibniz graph factorization of \emph{all} tri-differential components in the associator for $\star_{\text{aff}}$ mod $\bar{o}(\hbar^7)$ and we provide the data files of Leibniz graphs and their coefficients: see
~\ref{AppStarAffineAssoc7} on p.~\pageref{AppStarAffineAssoc7} below.
\end{proof}

The reduced affine star\/-\/product $\star_{\text{aff}}^{\text{red}}$ mod $\bar{o}(\hbar^7)$ is obtained from the affine star\/-\/product $\star_{\text{aff}}$ mod $\bar{o}(\hbar^7)$ by realizing the coefficient of $\zeta(3)^2/\pi^6$ as the Kontsevich graph expansion of a linear combination of Leibniz graphs with rational coefficients and, now that this combination does not contribute to either the star\/-\/product or its associator when restricted to any affine Poisson structure, by discarding this part of $\star_{\text{aff}}$ mod $\bar{o}(\hbar^7)$ proportional to $\zeta(3)^2/\pi^6$.
The same applies to many terms in the rational part of $\star_{\text{aff}}\text{ mod }\bar{o}(\hbar^7)$ which also assimilate to Leibniz graphs, see~\cite{affine7}.
In the reduced affine star\/-\/product $\star_{\text{aff}}^\text{red}$ mod $\bar{o}(\hbar^7)$ there remain only $326$ nonzero rational coefficients of Kontsevich graphs at $\hbar^k$ for $k=0,\ldots,7$ (in contrast with $1423$ nonzero (ir)rational coefficients at orders up to $\hbar^7$ in $\star_{\text{aff}}$ mod $\bar{o}(\hbar^7)$).

\begin{proof}[Proof scheme \textup{(}for the reduced affine star\/-\/product $\star_{\textup{aff}}^{\textup{red}}$ \textup{mod} $\bar{o}(\hbar^7)$\textup{)}]
\label{ProofReducedAffineStar7}
The associator for $\star_{\text{aff}}^{\text{red}}$ contains $95$ tri-differential orders at $\hbar^6$ and $161$ tri-differential orders at $\hbar^7$. 
We see that the associator $\operatorname{Assoc}(\star_{\text{aff}}^{\text{red}})$ mod $\bar{o}(\hbar^7)$ becomes much smaller than $\operatorname{Assoc}(\star_{\text{aff}})$ mod $\bar{o}(\hbar^7)$, now containing only $29371$ Kontsevich graphs instead of $59905$. 
But the work of the associativity mechanism for $\star_{\text{aff}}^{\text{red}}$ requires the use of the $1$st 
layer 
of Leibniz graphs much more often than it already was for the affine star\/-\/product $\star_{\text{aff}}$ mod $\bar{o}(\hbar^7)$ before the reduction.
Now, at orders $\leqslant 7$ in $\hbar$, new Leibniz graphs from the layer(s) beyond the $0$th are indispensable for the factorization of $114$ out of $336$ homogeneous tri-differential order components of the associator, see
~\ref{AppStarAffineReducedAssoc7}
where we list all these exceptional orders.
\end{proof}

\begin{rem}
We observe that the number $\zeta(3)^2/\pi^6$, not showing up in any restriction of the affine star\/-\/product $f \star_{\text{aff}} g$ mod $\bar{o}(\hbar^7)$ to an affine Poisson structure and any arguments $f,g \in
 A 
[[\hbar]]$, acts in effect as a placeholder of the Kontsevich graphs which, by contributing to the associator and then creating the Leibniz graphs by edge contraction, provide almost all of the Leibniz graphs needed for a factorization of the associator for $\star_{\text{aff}}$ mod $\bar{o}(\hbar^7)$ via the Jacobi identity.
When the $\zeta(3)^2/\pi^6$-part of $\star_{\text{aff}}$ mod $\bar{o}(\hbar^7)$ itself is eliminated by using the Jacobi identity for affine Poisson structures, the remaining $\star_{\text{aff}}^{\text{red}}$ mod $\bar{o}(\hbar^7)$ and its associator rely heavily on the use of higher layer(s) of Leibniz graphs for a factorization solution to be achieved.
\end{rem}


\begin{proposition}
\label{PropLeibnizOrder7Layer1}
The $0$th layer of Leibniz graphs is \emph{not enough} to provide a factorization of the associator for the (either affine or full) Kontsevich star\/-\/product at order $\hbar^7$%
, whereas, according to Proposition~\ref{PropAssoc6} above, the $0$th layer of Leibniz graphs \emph{was enough} at order $\hbar^6$ to factor the associator for the full star\/-\/product.
\end{proposition}


\begin{proof}[Constructive proof]\hangindent=-6.15cm \hangafter=-6
Consider the Leibniz graph (see figure)%
{\unitlength=1.8mm
\linethickness{0.4pt}
\begin{picture}(0.1,0.1)
(-61,14)
\put(-6.00,9.00){$L_1 = {}$}
\put(2.00,5.00){\circle*{1}}
\put(9.00,5.00){\circle*{1}}
\put(16.00,5.00){\circle*{1}}
\put(2.00,14.00){\circle*{1}}
\put(2.00,14.00){\vector(0,-1){8.33}}
\put(2.00,14.00){\vector(2,-1){5.33}}
\put(7.33,11){\circle*{1}}
\put(7.33,11){\vector(0,1){6}}
\put(7.33,11){\vector(1,-4){1.5}}
\put(5.50,8.33){\circle*{1}}
\put(5.50,8.33){\vector(1,-1){3.0}}
\put(5.50,8.33){\vector(-1,-1){3.0}}
\put(9.00,9.33){\circle*{1}}
\put(9.00,9.33){\vector(0,-1){3.7}}
\put(9.00,9.33){\vector(-3,-1){3.0}}
\put(11.50,8.33){\circle*{1}}
\put(11.50,8.33){\vector(-3,1){2.0}}
\put(11.50,8.33){\vector(-1,-2){1.7}}
\put(11.50,8.33){\vector(3,-2){4.3}}
\put(7.33,17.33){\circle*{1}}
\put(7.33,17.33){\vector(-2,-1){5.5}}
\put(7.33,17.33){\vector(2,-3){7.62}}
\end{picture}%
}%
on three sinks $0,1,2$, with $\tilde{n}=6$ aerial vertices, and with edges [(3, 2), (3, 7), (4, 1), (4, 8), (5, 1), (5, 3), (6, 1), (6, 2), (6, 4), (7, 0), (7, 5), (8, 0), (8, 1)]. 
This Leibniz graph is needed for the factorization of the tri-differential component of order $(2,4,2)$ at $\hbar^7$ in the associator for $\star_{\text{aff}}$ mod $\bar{o}(\hbar^7)$.
This graph appears only in the $1$st layer of Leibniz graphs, not in the $0$th layer, as we contract edges of Kontsevich's graphs on $n=7$ aerial vertices in the associator for $\star_{\text{aff}}$ mod $\bar{o}(\hbar^7)$, and as we expand the resulting Leibniz graphs to the old and possibly new Kontsevich graphs.\footnote{%
This Leibniz graph cannot originate from any Kontsevich graph in the associator itself (even with aerial vertex in\/-\/degree~$\geqslant 2$).
Namely, all candidate Kontsevich graphs are composite, with one of the factors having zero weight.}
This Leibniz graph created in the $1$st layer appears with coefficient $2/135$ in an iteratively found factorization of the associator.
The genuine Kontsevich weight of this Leibniz graph calculated by using the program \textsf{kontsevint} by E. Panzer is also nonzero: $w(L_1) = -3/128\cdot 
{\zeta(3)^2} / {\pi^6} + 31/725760$.
The actual coefficient of $L_1$ in the \emph{canonical} factorization of the associator, as guaranteed by the Formality Theorem, equals $w(L_1)$ multiplied by some nonzero rational constant. 
The discrepancy between the found rational value in \emph{some} solution and the (ir)rational value in Kontsevich's canonical solution is likely due to an identity between Leibniz graphs which expand to a zero sum of Kontsevich graphs (see \cite[Part I, \S 3.5.2]{BuringDisser}).
But anyway, based on this empiric evidence we conclude the proof.
\end{proof}

\noindent\textbf{Conclusion.}\quad 
The above iterative scheme gives us \emph{a} solution to the weak factorization problem: each tri-differential component $\mathsf{A}_{d_0 d_1 d_2}$ is factorized independently from the others, so that the coefficients of the Leibniz graphs are not yet constrained overall --\,over different components\,-- by the Shoikhet\/--\/Felder\/--\/Willwacher cyclic weight relations and other relations.
In particular, the above scheme does not guarantee that the found coefficients of Leibniz graphs are equal (up to the multiplicity and 
constants~$c_n$) to the genuine Kontsevich weights of those Leibniz graphs.
The above scheme provides the necessary minimum number of layers of Leibniz graphs, whereas the calculation of Kontsevich's genuine weights of Leibniz graphs is sufficient to build a solution (the canonical one) for the associator factorization problem.
We remember that there exist identities, i.e.\ sums of Leibniz graphs which expand to zero sums of Kontsevich graphs (here, in the associator); such identities could make unnecessary the use of a Leibniz graph with nonzero genuine weight from a (high number, in particular the last) layer.
Hypothetically it might be that any solution needs the $0$th and $1$st layers, hence they are ``necessary'', but Kontsevich's canonical solution stretches over the $0$th, $1$st and $2$nd layers, thus they are ``sufficient''.
The above scheme does not guarantee that the genuine Kontsevich weight of a Leibniz graph in the known associator's factorization at order $\hbar^7$ will definitely be equal (up to the multiplicity and 
constants $c_n$) to this Leibniz graph's coefficient in a solution found using the last necessary layer.
We conclude that the $1$st layer of Leibniz graphs becomes necessary at~$\hbar^{k\geqslant 7}$ for \emph{any} factorization of the associator (with harmonic propagators for the Kontsevich graph weights in~$\star$ in its authentic gauge from~\cite{MK97}). Such use of the $1$st and higher layers could start earlier, at orders $k<7$, for the factorization problem's canonical solution given by the Kontsevich weights $w(L)$ of Leibniz graphs.

\begin{open}
Over how many layers do the canonical 
Kontsevich solutions of associator's factorization problem stretch\,?
In particular what is the factorization guaranteed by the Formality theorem for orders $3$,\ $4$,\ $5$,\ $6$,\ $7$, which we have considered so far using solutions of the weak factorization problem\,?
\end{open}

\noindent\textbf{Acknowledgements.}\quad 
The second author is grateful to the organizers of international symposium on Quantum Theory and Symmetries (QTS12) on 24--28~July 2023 in CVUT Prague, Czech Republic.
A part of this research was done while the authors were visiting at the $\smash{\text{IH\'ES}}$ in Bures\/-\/sur\/-\/Yvette, France.
R.B.\ thanks E.\,Panzer for granting access to \textsf{kontsevint} software;
A.K.\ thanks G.\,Dito and M.~Kontsevich for helpful discussions.%
\footnote{%
The research of R.B.\ was supported by project~$5020$ at the Institute of Mathematics,
Johannes Gutenberg\/--\/Uni\-ver\-si\-t\"at Mainz and by CRC-326 grant GAUS `Geometry and Arithmetic of Uniformized Structures'.
The travel of A.K.\ was partially supported by project 135110 at the Bernoulli Institute, University of Groningen.
A.K.~is grateful to the $\smash{\text{IH\'ES}}$ for financial support and hospitality.}

This paper is extracted in part from the text~\cite{affine7} by the same authors; the authors thank colleagues and anonymous experts who acted as referees of this work.

\appendix

\section{Corrigendum: The encoding and use of Leibniz graphs}
\label{AppLeibniz}

\noindent%
Normal forms for Leibniz graphs with one
Jacobiator were introduced in~\cite[Definition~5]{JPCS2017}:
the idea was to re-use the normal form for Kontsevich graphs. Namely, the Jacobiator was expanded into the sum of three Kontsevich graphs (built of wedges), all the incoming arrows (to the top of the tripod) were formally directed to the top of the lower wedge in each Kontsevich graph, and then we found the normal forms of the resulting three Kontsevich graphs, while also
remembering where the internal edge of the Jacobiator is located in
those normal forms. The normal form of the Leibniz graph then was:
choose the minimal (w.r.t. base-\((m+n)\) numbers) Kontsevich graph
encoding, supplemented with the indication of the internal Jacobiator
edge. (Besides, it is necessary to pay attention to whether the internal Jacobiator edge is labeled Left or Right, in order to expand the Leibniz graph with the correct sign~\(\pm 1\).)

This definition, i.e.~the pair (Kontsevich graph, marked edge) is
unfortunately not a true normal form of the Leibniz graph. Namely, it
can happen that the resulting Kontsevich graph has an automorphism that
maps the marked edge elsewhere, to a new place in the graph.
Consequently, two isomorphic Leibniz graphs could have different
``normal forms'' (differing only by the marking where the internal
Jacobiator edge is). This led to a visible pathology, namely to
redundant parameters in the systems of equations: one and the same
Leibniz graph, encoded differently, acquired two unrelated coefficients.
Fortunately, the effect disappeared when Leibniz graphs were expanded to
sums of Kontsevich graphs and similar terms were collected.

In consequence, that normal form 
was abandoned in favor of inambiguous (and fast) description of Leibniz graphs by
using the \(\textsf{nauty}\) software~\cite{nauty}.

\twocolumn
\section{Factorization of associator $\Assoc(\star)\text{ mod }\bar{o}(\hbar^6)$ via Leibniz graphs}
\label{AppStarAssoc6}


\noindent%
First we present the certificate of vanishing for the rational part of the associator at $\hbar^6$ via the $0$th layer of Leibniz graphs.
That is, working over the extension $\mathbb{Q}[\zeta(3)^2/\pi^6]$, we first take the purely rational part of all the coefficients of Kontsevich graphs in the associator.

\begin{notation}\label{NotationExpandContract}
For each tri-differential order (with respect to the three arguments $f,g,h$ in the associator)
actually showing up in the sum of Kontsevich's graphs under study
we list that order $(i,j,k)$ itself, the number of Kontsevich graphs with nonzero coefficients in that order,
the number of Leibniz graphs which are instantly produced by contracting one internal edge in the already available Kontsevich graphs, and the number of new Kontsevich graphs (possibly zero of them) to which the so far reached Leibniz graphs expand (when the Jacobiator is expanded). 
If the attained Leibniz graphs are enough to realize, when expanded, the initially given sum of Kontsevich graphs in that tri-differential order, the contraction\/-\/expansion stops.
If not, we repeat the iteration(s) until the initially given sum is successfully realized by the Leibniz graph expansions for the first and/or higher layers of neighbors. 
The program \textsf{gcaops} (\emph{Graph Complex Action on Poisson Structures} by R.~Buring) writes \verb"True" as soon as the fact of factorization is established.
\end{notation}

The vanishing of rational part of associator at $\hbar^6$, via $0$th layer Leibniz graphs:\\
{\tiny
\begin{verbatim}
Number of Kontsevich graphs: 290243
Number of differential orders: 105
(3, 1, 4): 449K -> +220L -> +26K
True
(2, 2, 4): 829K -> +424L -> +71K
True
(2, 1, 4): 1524K -> +780L -> +115K
True
(1, 3, 4): 443K -> +220L -> +32K
True
(1, 2, 4): 1515K -> +780L -> +124K
True
(1, 1, 4): 2315K -> +1135L -> +281K
True
(3, 2, 4): 208K -> +98L -> +17K
True
(2, 3, 4): 203K -> +98L -> +22K
True
(1, 4, 4): 75K -> +36L -> +11K
True
(4, 1, 4): 82K -> +36L -> +4K
True
(4, 2, 4): 32K -> +14L -> +7K
True
(3, 3, 4): 38K -> +16L -> +7K
True
(2, 4, 4): 32K -> +14L -> +7K
True
(4, 2, 3): 208K -> +98L -> +17K
True
(4, 1, 3): 449K -> +220L -> +26K
True
(3, 3, 3): 362K -> +175L -> +30K
True
(3, 2, 3): 1424K -> +810L -> +161K
True
(3, 1, 3): 2612K -> +1475L -> +199K
True
(2, 4, 3): 199K -> +97L -> +26K
True
(2, 3, 3): 1423K -> +810L -> +162K
True
(2, 2, 3): 4984K -> +2947L -> +451K
True
(2, 1, 3): 7702K -> +4353L -> +618K
True
(1, 4, 3): 417K -> +215L -> +58K
True
(1, 3, 3): 2583K -> +1469L -> +216K
True
(1, 2, 3): 7659K -> +4350L -> +661K
True
(1, 1, 3): 10263K -> +5295L -> +1217K
True
(4, 3, 3): 38K -> +16L -> +7K
True
(3, 4, 3): 38K -> +16L -> +7K
True
(2, 5, 3): 10K -> +5L -> +5K
True
(1, 5, 3): 36K -> +14L -> +6K
True
(5, 2, 3): 15K -> +5L -> +0K
True
(5, 1, 3): 41K -> +14L -> +1K
True
(5, 3, 3): 3K -> +1L -> +0K
True
(4, 4, 3): 3K -> +1L -> +0K
True
(3, 5, 3): 3K -> +1L -> +0K
True
(1, 5, 4): 7K -> +4L -> +5K
True
(5, 1, 4): 12K -> +4L -> +0K
True
(5, 2, 4): 3K -> +1L -> +0K
True
(4, 3, 4): 3K -> +1L -> +0K
True
(3, 4, 4): 3K -> +1L -> +0K
True
(2, 5, 4): 3K -> +1L -> +0K
True
(4, 2, 2): 829K -> +424L -> +71K
True
(4, 1, 2): 1524K -> +780L -> +115K
True
(3, 3, 2): 1423K -> +810L -> +162K
True
(3, 2, 2): 4984K -> +2947L -> +451K
True
(3, 1, 2): 7702K -> +4353L -> +618K
True
(2, 3, 2): 4779K -> +2908L -> +631K
True
(2, 2, 2): 14046K -> +8416L -> +1618K
True
(2, 1, 2): 18894K -> +10298L -> +1904K
True
(1, 4, 2): 1338K -> +752L -> +266K
True
(1, 3, 2): 7297K -> +4282L -> +956K
True
(1, 2, 2): 19000K -> +10368L -> +1796K
True
(1, 1, 2): 22789K -> +10742L -> +2227K
True
(4, 3, 2): 203K -> +98L -> +22K
True
(3, 4, 2): 199K -> +97L -> +26K
True
(2, 4, 2): 758K -> +421L -> +141K
True
(1, 5, 2): 96K -> +43L -> +33K
True
(5, 2, 2): 53K -> +18L -> +1K
True
(5, 1, 2): 121K -> +43L -> +8K
True
(5, 3, 2): 15K -> +5L -> +0K
True
(4, 4, 2): 32K -> +14L -> +7K
True
(2, 5, 2): 43K -> +18L -> +11K
True
(3, 5, 2): 10K -> +5L -> +5K
True
(4, 2, 1): 1515K -> +780L -> +124K
True
(4, 1, 1): 2315K -> +1135L -> +281K
True
(3, 3, 1): 2583K -> +1469L -> +216K
True
(3, 2, 1): 7659K -> +4350L -> +661K
True
(3, 1, 1): 10263K -> +5295L -> +1217K
True
(2, 4, 1): 1338K -> +752L -> +266K
True
(2, 3, 1): 7297K -> +4282L -> +956K
True
(2, 2, 1): 19000K -> +10368L -> +1796K
True
(2, 1, 1): 22789K -> +10742L -> +2227K
True
(1, 4, 1): 2223K -> +1135L -> +373K
True
(1, 3, 1): 10068K -> +5290L -> +1412K
True
(1, 2, 1): 22591K -> +10736L -> +2424K
True
(1, 1, 1): 23814K -> +9358L -> +2709K
True
(4, 3, 1): 443K -> +220L -> +32K
True
(3, 4, 1): 417K -> +215L -> +58K
True
(2, 5, 1): 96K -> +43L -> +33K
True
(1, 5, 1): 234K -> +81L -> +8K
True
(5, 2, 1): 121K -> +43L -> +8K
True
(5, 1, 1): 234K -> +81L -> +8K
True
(5, 3, 1): 41K -> +14L -> +1K
True
(4, 4, 1): 75K -> +36L -> +11K
True
(3, 5, 1): 36K -> +14L -> +6K
True
(5, 4, 1): 12K -> +4L -> +0K
True
(5, 4, 2): 3K -> +1L -> +0K
True
(4, 5, 2): 3K -> +1L -> +0K
True
(1, 1, 5): 234K -> +81L -> +8K
True
(5, 1, 5): 3K -> +1L -> +0K
True
(4, 2, 5): 3K -> +1L -> +0K
True
(3, 3, 5): 3K -> +1L -> +0K
True
(2, 4, 5): 3K -> +1L -> +0K
True
(1, 5, 5): 3K -> +1L -> +0K
True
(2, 1, 5): 121K -> +43L -> +8K
True
(1, 2, 5): 121K -> +43L -> +8K
True
(4, 1, 5): 12K -> +4L -> +0K
True
(3, 2, 5): 15K -> +5L -> +0K
True
(2, 3, 5): 15K -> +5L -> +0K
True
(1, 4, 5): 12K -> +4L -> +0K
True
(3, 1, 5): 41K -> +14L -> +1K
True
(2, 2, 5): 53K -> +18L -> +1K
True
(1, 3, 5): 41K -> +14L -> +1K
True
(4, 5, 1): 7K -> +4L -> +5K
True
(5, 5, 1): 3K -> +1L -> +0K
True
\end{verbatim}
}

Let us remember that in every tri-differential order we split the sums of Kontsevich graphs with coefficients from $\mathbb{Q}[\zeta(3)^2/\pi^6]$ into the part with purely rational coefficients and $\zeta(3)^2/\pi^6$ times another sum of Kontsevich graphs with purely rational coefficients.
We now inspect a realization of the $\zeta(3)^2/\pi^6$-part by using the layer(s) of Leibniz graphs which we produce from that part alone, that is, regardless of the Kontsevich graphs also available in the (rational part of the) same tri-differential order.
This is why the first layer of neighbors is (fictitiously) referred to in a factorization of the associator mod $\bar{o}(\hbar^6)$.

The vanishing of the $\zeta(3)^2/\pi^6$-part of the star\/-\/product's associator at $\hbar^6$, via $0$th and $1$st layer Leibniz graphs (see Notation on p.~\pageref{NotationExpandContract}):

{\tiny
\begin{verbatim}
Number of Kontsevich graphs: 194060
Number of differential orders: 28
(2, 1, 3): 4987K -> +3481L -> +1447K
True
(1, 1, 3): 8899K -> +5099L -> +2489K -> +648L -> +201K
True
(3, 1, 3): 732K -> +592L -> +488K
True
(2, 2, 2): 6240K -> +5047L -> +4038K
True
(3, 1, 2): 4987K -> +3481L -> +1447K
True
(2, 1, 2): 16100K -> +9665L -> +2912K -> +575L -> +84K
True
(1, 2, 2): 14200K -> +9001L -> +4579K -> +1224L -> +311K
True
(1, 1, 2): 21813K -> +10699L -> +3178K
True
(3, 2, 2): 988K -> +904L -> +844K
True
(4, 1, 2): 520K -> +392L -> +299K
True
(4, 1, 1): 1363K -> +876L -> +1061K
True
(3, 2, 1): 4173K -> +3076L -> +2051K
True
(3, 1, 1): 8899K -> +5099L -> +2489K -> +648L -> +201K
True
(2, 3, 1): 2620K -> +2084L -> +2797K
True
(2, 2, 1): 14200K -> +9001L -> +4579K -> +1224L -> +311K
True
(2, 1, 1): 21813K -> +10699L -> +3178K
True
(1, 3, 1): 5913K -> +3834L -> +4472K -> +1749L -> +1122K
True
(1, 2, 1): 20238K -> +10386L -> +4612K
True
(1, 1, 1): 23331K -> +9345L -> +3180K
True
(4, 2, 1): 520K -> +392L -> +299K
True
(3, 3, 1): 670K -> +566L -> +487K
True
(1, 1, 4): 1363K -> +876L -> +1061K
True
(1, 3, 2): 2620K -> +2084L -> +2797K
True
(1, 2, 3): 4173K -> +3076L -> +2051K
True
(2, 1, 4): 520K -> +392L -> +299K
True
(1, 2, 4): 520K -> +392L -> +299K
True
(2, 2, 3): 988K -> +904L -> +844K
True
(1, 3, 3): 670K -> +566L -> +487K
True
\end{verbatim}
}

Let us resolve an apparent violation of our earlier claim. 
In earnest, the entire associator up to $\bar{o}(h^6)$ is realized by using only the $0$th layer of Leibniz graphs.
Indeed, to represent the sum near $\zeta(3)^2/\pi^6$, let us use all the Kontsevich graphs which showed up in each tri-differential order (but maybe did not show with nonzero rational coefficients in the co-multiple of $\zeta(3)^2/\pi^6$).
We thus detect that the $0$th layer is always enough at $\hbar^6$.

The vanishing of star\/-\/product's associator at $\hbar^6$ (without splitting into rational and $\zeta(3)^2/\pi^6$ parts) at exceptional differential orders, via $0$th layer Leibniz graphs:

{\tiny
\begin{verbatim}
Number of Kontsevich graphs: 290305
(1, 3, 1): 10068K -> +5290L -> +1412K
True
(3, 1, 1): 10264K -> +5295L -> +1216K
True
(1, 1, 3): 10264K -> +5295L -> +1216K
True
(2, 2, 1): 19006K -> +10368L -> +1790K
True
(1, 2, 2): 19006K -> +10368L -> +1790K
True
(2, 1, 2): 18901K -> +10298L -> +1897K
True
\end{verbatim}
}

Hence Kontsevich's \(\star\) mod \(\bar{o}(\hbar^6)\) is associative;
its associator up to \(\bar{o}(\hbar^6)\) is a sum of Leibniz graphs
from the \(0\)th layer.

\section{Beyond the 0th layer of Leibniz graphs in factorization of $\Assoc(\star_{\text{aff}})\text{ mod }\bar{o}(\hbar^7)$}
\label{AppStarAffineAssoc7}

\noindent%
Using the same contraction\/-\/expansion method as in
~\ref{AppStarAssoc6} (now adapted to affine Kontsevich graphs with in-degree bound $\leqslant 1$ for aerial vertices and to affine Leibniz graphs with in-degree bound $\leqslant 2$ for the trident vertex and $\leqslant 1$ for all other aerial vertices)
we obtain the certificate of vanishing for the associator of $\star_{\text{aff}}$ at~$\hbar^7$.
Now, Leibniz graphs from beyond the $0$th layer must be used (and this does not stem from any splitting of the associator into the rational and $\zeta(3)^2/\pi^6$ parts).
The source files with coefficients and graph encodings are stored externally:

$\bullet$\quad The part proportional to $\hbar^7$ in the associator for $\star_{\text{aff}}$ mod $\bar{o}(\hbar^7)$, after we filter only the Kontsevich graphs with aerial vertex in-degrees $\leqslant 1$, is stored in the file
{\small\begin{center}
	\verb"https://www.rburing.nl/gcaops/"
	\verb"affine_assoc7_7.txt"
\end{center}}
We know that $49621$ Kontsevich graphs truly survive into the associator at order~$\hbar^7$ for the genuine affine Kontsevich star\/-\/product modulo $\bar{o}(\hbar^7)$ with harmonic propagators in the graph weights.

$\bullet$\quad A linear combination of Leibniz graphs with coefficients from $\mathbb{Q}[\zeta(3)^2/\pi^6]$ which suffices to realize the respective tri-differential order at~$\hbar^7$ in the associator for $\star_{\text{aff}}$ mod $\bar{o}(\hbar^7)$ is contained in the pair of files with graph encodings and coefficients respectively 
{\small\begin{center}
\end{center}}
We detect that at $\hbar^7$ the $0$th layer of Leibniz graphs is not enough to build a factorization of $\star_{\text{aff}}$'s associator. Here is a sample from the \textsf{gcaops} output, which itself is contained in the file
{\small\begin{center}
\end{center}}
\noindent and in \cite[pp.~118--120]{BuringDisser}:

{\tiny\begin{verbatim}
(5, 2, 1): 274K -> +346L -> +20K
(6, 4, 1): 18K -> +10L -> +1K
(5, 5, 1): 41K -> +30L -> +6K
(4, 6, 1): 16K -> +10L -> +3K
(2, 3, 3): 2294K -> +3584L -> +221K -> +123L -> +35K
(3, 2, 3): 2331K -> +3603L -> +191K -> +106L -> +30K
(3, 3, 2): 2294K -> +3584L -> +221K -> +123L -> +35K
(2, 4, 2): 1246K -> +2041L -> +273K -> +111L -> +23K
(4, 2, 2): 1431K -> +2111L -> +98K
(2, 2, 3): 1095K -> +1967L -> +47K
(3, 3, 1): 616K -> +1091L -> +32K
(2, 4, 1): 353K -> +636L -> +49K
\end{verbatim}

}

\section{The 1st layer of Leibniz graphs 
much used in any factorization of the associator $\Assoc(\star_{\text{aff}}^{\text{red}})\text{ mod }\bar{o}(\hbar^7)$ for the reduced affine star\/-\/product}
\label{AppStarAffineReducedAssoc7}

\noindent%
Let us inspect how the associator of $\star_{\text{aff}}^{\text{red}}$ mod $\bar{o}(\hbar^7)$ is realized by sums of Leibniz graphs; we proceed along the powers of the deformation parameter $\hbar$.
In every term, we run over the tri-differential orders where Kontsevich graphs show up with nonzero coefficients from $\mathbb{Q}$.
We see that the first layer of Leibniz graphs is much used (for the first time it is needed in the order $(1,2,1)$ at $\hbar^4$).
At the same time, the first layer is sufficient to build all the factorizations up to~$\bar{o}(\hbar^7)$.


{\tiny
\begin{verbatim}
h^2:
Number of differential orders: 1
(1, 1, 1): 3K -> +1L -> +0K
True
h^3:
Number of differential orders: 7
(1, 1, 2): 4K -> +3L -> +2K
True
(2, 1, 2): 3K -> +1L -> +0K
True
(1, 2, 2): 3K -> +1L -> +0K
True
(1, 1, 1): 2K -> +2L -> +2K
True
(2, 2, 1): 3K -> +1L -> +0K
True
(1, 2, 1): 4K -> +3L -> +2K
True
(2, 1, 1): 4K -> +3L -> +2K
True
h^4:
Number of differential orders: 22
(1, 1, 2): 4K -> +6L -> +5K
True
(1, 1, 3): 8K -> +6L -> +2K
True
(2, 1, 2): 18K -> +17L -> +8K
True
(1, 2, 2): 18K -> +18L -> +8K
True
(2, 1, 3): 7K -> +4L -> +2K
True
(1, 2, 3): 7K -> +4L -> +2K
True
(3, 1, 3): 3K -> +1L -> +0K
True
(2, 2, 3): 3K -> +1L -> +0K
True
(1, 3, 3): 3K -> +1L -> +0K
True
(3, 1, 2): 7K -> +4L -> +2K
True
(2, 2, 2): 16K -> +12L -> +8K
True
(1, 3, 2): 7K -> +4L -> +2K
True
(3, 1, 1): 8K -> +6L -> +2K
True
(2, 2, 1): 20K -> +18L -> +6K
True
(1, 3, 1): 9K -> +6L -> +1K
True
(2, 1, 1): 4K -> +6L -> +5K
True
(3, 2, 2): 3K -> +1L -> +0K
True
(2, 3, 2): 3K -> +1L -> +0K
True
(1, 2, 1): 2K -> +2L -> +4K -> +8L -> +6K
True
(2, 3, 1): 7K -> +4L -> +2K
True
(3, 2, 1): 7K -> +4L -> +2K
True
(3, 3, 1): 3K -> +1L -> +0K
True
h^5:
Number of differential orders: 50
(2, 1, 2): 8K -> +18L -> +18K
True
(1, 2, 2): 12K -> +30L -> +23K
True
(3, 1, 2): 40K -> +58L -> +35K
True
(2, 2, 2): 77K -> +110L -> +58K -> +37L -> +19K
True
(2, 1, 3): 40K -> +55L -> +33K
True
(1, 2, 3): 43K -> +61L -> +30K
True
(2, 2, 3): 65K -> +59L -> +23K
True
(3, 2, 2): 71K -> +62L -> +21K
True
(2, 3, 2): 68K -> +60L -> +20K
True
(1, 3, 2): 36K -> +50L -> +25K -> +11L -> +8K
True
(3, 1, 3): 31K -> +28L -> +13K
True
(1, 3, 3): 29K -> +28L -> +15K
True
(3, 2, 3): 22K -> +14L -> +8K
True
(2, 3, 3): 22K -> +14L -> +8K
True
(3, 1, 4): 7K -> +4L -> +2K
True
(2, 2, 4): 10K -> +5L -> +2K
True
(1, 3, 4): 7K -> +4L -> +2K
True
(4, 2, 3): 3K -> +1L -> +0K
True
(3, 3, 3): 3K -> +1L -> +0K
True
(4, 1, 3): 7K -> +4L -> +2K
True
(2, 4, 3): 3K -> +1L -> +0K
True
(1, 4, 3): 7K -> +4L -> +2K
True
(4, 1, 4): 3K -> +1L -> +0K
True
(3, 2, 4): 3K -> +1L -> +0K
True
(2, 3, 4): 3K -> +1L -> +0K
True
(1, 4, 4): 3K -> +1L -> +0K
True
(2, 1, 4): 14K -> +10L -> +4K
True
(1, 2, 4): 14K -> +10L -> +4K
True
(4, 2, 2): 10K -> +5L -> +2K
True
(3, 3, 2): 22K -> +14L -> +8K
True
(2, 4, 2): 10K -> +5L -> +2K
True
(4, 1, 2): 12K -> +9L -> +4K
True
(1, 4, 2): 13K -> +9L -> +3K
True
(1, 1, 4): 10K -> +9L -> +5K -> +1L -> +0K
True
(4, 2, 1): 12K -> +9L -> +4K
True
(4, 1, 1): 10K -> +9L -> +5K -> +1L -> +0K
True
(3, 3, 1): 33K -> +29L -> +11K
True
(3, 2, 1): 44K -> +61L -> +31K
True
(2, 4, 1): 13K -> +9L -> +3K
True
(2, 3, 1): 35K -> +50L -> +26K -> +11L -> +8K
True
(1, 4, 1): 8K -> +10L -> +7K
True
(1, 1, 3): 4K -> +10L -> +10K
True
(3, 1, 1): 4K -> +10L -> +10K
True
(2, 2, 1): 12K -> +30L -> +23K
True
(1, 3, 1): 6K -> +16L -> +12K
True
(4, 3, 2): 3K -> +1L -> +0K
True
(3, 4, 2): 3K -> +1L -> +0K
True
(3, 4, 1): 7K -> +4L -> +2K
True
(4, 3, 1): 7K -> +4L -> +2K
True
(4, 4, 1): 3K -> +1L -> +0K
True
h^6:
Number of differential orders: 95
(2, 2, 2): 46K -> +111L -> +104K -> +202L -> +130K
True
(1, 3, 2): 32K -> +80L -> +57K -> +96L -> +60K
True
(3, 1, 2): 27K -> +66L -> +74K -> +135L -> +66K
True
(3, 2, 2): 216K -> +423L -> +259K -> +205L -> +76K
True
(2, 3, 2): 225K -> +448L -> +239K
True
(4, 1, 2): 64K -> +109L -> +74K -> +53L -> +25K
True
(4, 2, 2): 164K -> +201L -> +90K -> +19L -> +10K
True
(3, 3, 2): 268K -> +362L -> +165K -> +55L -> +23K
True
(2, 2, 3): 219K -> +422L -> +251K -> +205L -> +83K
True
(1, 3, 3): 136K -> +246L -> +121K -> +73L -> +33K
True
(3, 1, 3): 113K -> +203L -> +141K -> +112L -> +36K
True
(3, 2, 3): 266K -> +358L -> +171K -> +57L -> +18K
True
(2, 3, 3): 270K -> +357L -> +160K -> +59L -> +26K
True
(4, 1, 3): 87K -> +106L -> +56K -> +18L -> +8K
True
(4, 2, 3): 99K -> +83L -> +30K
True
(3, 3, 3): 169K -> +139L -> +44K
True
(2, 4, 2): 147K -> +182L -> +78K -> +18L -> +11K
True
(1, 4, 2): 76K -> +126L -> +58K -> +32L -> +20K
True
(4, 3, 2): 99K -> +83L -> +30K
True
(3, 4, 2): 92K -> +78L -> +29K
True
(5, 1, 2): 24K -> +18L -> +8K -> +1L -> +0K
True
(5, 2, 2): 25K -> +17L -> +8K
True
(1, 4, 3): 74K -> +88L -> +41K -> +17L -> +10K
True
(2, 4, 3): 90K -> +77L -> +31K
True
(5, 2, 3): 14K -> +7L -> +2K
True
(4, 3, 3): 28K -> +16L -> +8K
True
(3, 4, 3): 28K -> +16L -> +8K
True
(3, 2, 4): 97K -> +81L -> +30K
True
(2, 3, 4): 95K -> +80L -> +32K
True
(1, 4, 4): 36K -> +31L -> +15K
True
(4, 1, 4): 41K -> +33L -> +14K
True
(4, 2, 4): 22K -> +14L -> +8K
True
(3, 3, 4): 28K -> +16L -> +8K
True
(2, 4, 4): 22K -> +14L -> +8K
True
(2, 5, 3): 10K -> +5L -> +2K
True
(1, 5, 3): 16K -> +10L -> +3K
True
(5, 1, 3): 18K -> +12L -> +5K
True
(5, 3, 3): 3K -> +1L -> +0K
True
(4, 4, 3): 3K -> +1L -> +0K
True
(3, 5, 3): 3K -> +1L -> +0K
True
(1, 5, 4): 7K -> +4L -> +2K
True
(5, 2, 4): 3K -> +1L -> +0K
True
(4, 3, 4): 3K -> +1L -> +0K
True
(3, 4, 4): 3K -> +1L -> +0K
True
(2, 5, 4): 3K -> +1L -> +0K
True
(2, 2, 4): 167K -> +202L -> +87K -> +18L -> +10K
True
(1, 3, 4): 91K -> +114L -> +52K
True
(3, 1, 4): 88K -> +106L -> +55K -> +18L -> +8K
True
(2, 5, 2): 20K -> +13L -> +5K
True
(1, 5, 2): 20K -> +17L -> +8K
True
(5, 3, 2): 14K -> +7L -> +2K
True
(4, 4, 2): 22K -> +14L -> +8K
True
(3, 5, 2): 10K -> +5L -> +2K
True
(1, 2, 4): 80K -> +129L -> +62K
True
(2, 1, 4): 64K -> +109L -> +74K -> +53L -> +25K
True
(4, 2, 1): 80K -> +129L -> +62K
True
(3, 3, 1): 134K -> +244L -> +125K -> +75L -> +30K
True
(2, 4, 1): 76K -> +126L -> +58K -> +32L -> +20K
True
(4, 3, 1): 91K -> +113L -> +52K
True
(3, 4, 1): 75K -> +92L -> +42K -> +13L -> +8K
True
(2, 5, 1): 20K -> +17L -> +8K
True
(1, 5, 1): 18K -> +15L -> +3K
True
(5, 2, 1): 24K -> +18L -> +8K -> +1L -> +0K
True
(5, 1, 1): 14K -> +15L -> +7K
True
(5, 3, 1): 18K -> +12L -> +5K
True
(4, 4, 1): 39K -> +32L -> +12K
True
(3, 5, 1): 16K -> +10L -> +3K
True
(1, 2, 3): 38K -> +92L -> +74K -> +112L -> +58K
True
(2, 1, 3): 28K -> +68L -> +75K -> +134L -> +65K
True
(3, 2, 1): 38K -> +93L -> +71K -> +109L -> +57K
True
(2, 3, 1): 32K -> +80L -> +57K -> +96L -> +60K
True
(1, 4, 1): 8K -> +18L -> +18K -> +34L -> +22K
True
(4, 1, 1): 12K -> +25L -> +19K
True
(5, 4, 2): 3K -> +1L -> +0K
True
(4, 5, 2): 3K -> +1L -> +0K
True
(1, 1, 4): 12K -> +25L -> +19K
True
(4, 1, 5): 9K -> +5L -> +2K
True
(3, 2, 5): 14K -> +7L -> +2K
True
(2, 3, 5): 14K -> +7L -> +2K
True
(1, 4, 5): 9K -> +5L -> +2K
True
(3, 1, 5): 20K -> +13L -> +5K
True
(2, 2, 5): 29K -> +19L -> +8K
True
(1, 3, 5): 20K -> +13L -> +5K
True
(2, 1, 5): 24K -> +18L -> +8K -> +1L -> +0K
True
(1, 2, 5): 24K -> +18L -> +8K -> +1L -> +0K
True
(1, 1, 5): 14K -> +15L -> +7K
True
(5, 1, 5): 3K -> +1L -> +0K
True
(4, 2, 5): 3K -> +1L -> +0K
True
(3, 3, 5): 3K -> +1L -> +0K
True
(2, 4, 5): 3K -> +1L -> +0K
True
(1, 5, 5): 3K -> +1L -> +0K
True
(5, 1, 4): 9K -> +5L -> +2K
True
(4, 5, 1): 7K -> +4L -> +2K
True
(5, 4, 1): 9K -> +5L -> +2K
True
(5, 5, 1): 3K -> +1L -> +0K
True
h^7:
Number of differential orders: 161
(3, 2, 2): 87K -> +314L -> +353K -> +915L -> +533K
True
(2, 3, 2): 107K -> +372L -> +341K -> +835L -> +468K
True
(3, 3, 2): 762K -> +1930L -> +1258K -> +1463L -> +475K
True
(2, 4, 2): 509K -> +1243L -> +699K -> +669L -> +269K
True
(4, 2, 2): 455K -> +1114L -> +766K -> +836L -> +295K
True
(4, 3, 2): 777K -> +1378L -> +672K -> +361L -> +125K
True
(3, 4, 2): 799K -> +1399L -> +640K -> +330L -> +113K
True
(5, 2, 2): 288K -> +438L -> +236K -> +109L -> +40K
True
(5, 3, 2): 264K -> +296L -> +133K -> +22L -> +5K
True
(4, 4, 2): 411K -> +506L -> +214K -> +61L -> +26K
True
(3, 2, 3): 638K -> +1612L -> +1293K -> +1732L -> +570K
True
(2, 3, 3): 760K -> +1919L -> +1251K -> +1457L -> +494K
True
(4, 2, 3): 736K -> +1333L -> +722K -> +421L -> +130K
True
(3, 3, 3): 1252K -> +2303L -> +1110K -> +658L -> +205K
True
(2, 4, 3): 795K -> +1392L -> +636K -> +336L -> +120K
True
(4, 3, 3): 632K -> +785L -> +334K -> +88L -> +30K
True
(3, 4, 3): 629K -> +781L -> +326K -> +86L -> +31K
True
(5, 2, 3): 262K -> +295L -> +135K -> +23L -> +5K
True
(5, 3, 3): 135K -> +107L -> +42K
True
(4, 4, 3): 208K -> +167L -> +57K
True
(3, 5, 2): 231K -> +259L -> +105K -> +21L -> +11K
True
(2, 5, 2): 297K -> +447L -> +196K -> +71L -> +33K
True
(5, 4, 2): 107K -> +86L -> +32K
True
(4, 5, 2): 101K -> +82L -> +30K
True
(6, 2, 2): 42K -> +30L -> +13K -> +1L -> +0K
True
(6, 3, 2): 28K -> +18L -> +8K
True
(2, 5, 3): 227K -> +254L -> +107K -> +26L -> +13K
True
(3, 5, 3): 117K -> +96L -> +40K
True
(5, 4, 3): 32K -> +18L -> +8K
True
(4, 5, 3): 28K -> +16L -> +8K
True
(4, 2, 4): 402K -> +494L -> +219K -> +63L -> +20K
True
(3, 3, 4): 641K -> +791L -> +330K -> +85L -> +29K
True
(2, 4, 4): 408K -> +499L -> +212K -> +68L -> +31K
True
(5, 2, 4): 103K -> +83L -> +32K
True
(4, 3, 4): 214K -> +169L -> +57K
True
(3, 4, 4): 210K -> +167L -> +57K
True
(2, 5, 4): 100K -> +81L -> +31K
True
(5, 3, 4): 32K -> +18L -> +8K
True
(4, 4, 4): 34K -> +18L -> +8K
True
(3, 5, 4): 28K -> +16L -> +8K
True
(3, 6, 3): 10K -> +5L -> +2K
True
(2, 6, 3): 23K -> +14L -> +5K
True
(6, 3, 3): 10K -> +5L -> +2K
True
(6, 2, 3): 28K -> +18L -> +8K
True
(6, 4, 3): 3K -> +1L -> +0K
True
(5, 5, 3): 3K -> +1L -> +0K
True
(4, 6, 3): 3K -> +1L -> +0K
True
(2, 6, 4): 10K -> +5L -> +2K
True
(6, 2, 4): 10K -> +5L -> +2K
True
(6, 3, 4): 3K -> +1L -> +0K
True
(5, 4, 4): 3K -> +1L -> +0K
True
(4, 5, 4): 3K -> +1L -> +0K
True
(3, 6, 4): 3K -> +1L -> +0K
True
(3, 2, 4): 733K -> +1340L -> +727K -> +411L -> +127K
True
(2, 3, 4): 779K -> +1389L -> +675K -> +352L -> +122K
True
(3, 6, 2): 23K -> +14L -> +5K
True
(2, 6, 2): 36K -> +27L -> +11K
True
(6, 4, 2): 10K -> +5L -> +2K
True
(5, 5, 2): 24K -> +15L -> +8K
True
(4, 6, 2): 10K -> +5L -> +2K
True
(2, 2, 4): 456K -> +1117L -> +772K -> +848L -> +289K
True
(5, 3, 1): 160K -> +243L -> +144K -> +78L -> +36K
True
(5, 2, 1): 106K -> +224L -> +152K -> +124L -> +58K
True
(4, 4, 1): 259K -> +452L -> +238K -> +124L -> +48K
True
(4, 3, 1): 289K -> +712L -> +422K -> +353L -> +134K
True
(3, 5, 1): 158K -> +240L -> +122K -> +52L -> +23K
True
(3, 4, 1): 280K -> +681L -> +392K -> +359L -> +142K
True
(2, 5, 1): 107K -> +226L -> +136K -> +94L -> +42K
True
(5, 4, 1): 114K -> +133L -> +63K
True
(4, 5, 1): 86K -> +99L -> +50K -> +17L -> +7K
True
(3, 6, 1): 24K -> +20L -> +10K
True
(2, 6, 1): 30K -> +28L -> +12K
True
(6, 3, 1): 28K -> +21L -> +10K -> +1L -> +0K
True
(6, 2, 1): 34K -> +32L -> +16K -> +2L -> +1K
True
(6, 4, 1): 15K -> +10L -> +4K
True
(5, 5, 1): 35K -> +30L -> +12K
True
(4, 6, 1): 16K -> +10L -> +3K
True
(2, 2, 3): 87K -> +318L -> +361K -> +922L -> +510K
True
(4, 2, 1): 42K -> +139L -> +115K -> +221L -> +128K
True
(3, 3, 1): 78K -> +276L -> +227K -> +453L -> +240K
True
(2, 4, 1): 49K -> +164L -> +125K -> +230L -> +135K
True
(1, 5, 2): 111K -> +233L -> +134K -> +87L -> +40K
True
(5, 1, 2): 94K -> +192L -> +162K -> +157L -> +60K
True
(1, 5, 3): 156K -> +237L -> +123K -> +55L -> +24K
True
(5, 1, 3): 147K -> +232L -> +160K -> +94L -> +42K
True
(6, 5, 2): 3K -> +1L -> +0K
True
(5, 6, 2): 3K -> +1L -> +0K
True
(1, 6, 2): 30K -> +28L -> +12K
True
(6, 1, 2): 34K -> +32L -> +16K -> +2L -> +1K
True
(1, 6, 3): 24K -> +20L -> +10K
True
(6, 1, 3): 28K -> +21L -> +10K -> +1L -> +0K
True
(3, 1, 4): 204K -> +481L -> +461K -> +600L -> +201K
True
(1, 3, 4): 287K -> +703L -> +423K -> +368L -> +135K
True
(4, 1, 5): 111K -> +127L -> +67K -> +15L -> +3K
True
(3, 2, 5): 261K -> +296L -> +135K -> +21L -> +4K
True
(2, 3, 5): 263K -> +297L -> +133K -> +20L -> +4K
True
(1, 4, 5): 113K -> +132L -> +65K
True
(3, 1, 5): 145K -> +231L -> +160K -> +94L -> +42K
True
(2, 2, 5): 296K -> +444L -> +234K -> +105L -> +38K
True
(1, 3, 5): 159K -> +243L -> +145K -> +77L -> +34K
True
(4, 1, 4): 231K -> +410L -> +265K -> +177L -> +72K
True
(1, 4, 4): 258K -> +451L -> +239K -> +125L -> +48K
True
(2, 1, 5): 93K -> +191L -> +162K -> +152L -> +55K
True
(1, 2, 5): 104K -> +223L -> +149K -> +119L -> +57K
True
(4, 1, 3): 202K -> +470L -> +456K -> +601L -> +205K
True
(1, 4, 3): 283K -> +693L -> +396K -> +355L -> +143K
True
(5, 1, 6): 9K -> +5L -> +2K
True
(4, 2, 6): 10K -> +5L -> +2K
True
(3, 3, 6): 10K -> +5L -> +2K
True
(2, 4, 6): 10K -> +5L -> +2K
True
(1, 5, 6): 9K -> +5L -> +2K
True
(4, 1, 6): 17K -> +11L -> +4K
True
(3, 2, 6): 32K -> +20L -> +8K
True
(2, 3, 6): 32K -> +20L -> +8K
True
(1, 4, 6): 17K -> +11L -> +4K
True
(5, 1, 5): 35K -> +29L -> +12K
True
(4, 2, 5): 104K -> +83L -> +29K
True
(3, 3, 5): 135K -> +108L -> +42K
True
(2, 4, 5): 107K -> +86L -> +30K
True
(1, 5, 5): 36K -> +30L -> +11K
True
(3, 1, 6): 28K -> +21L -> +10K -> +1L -> +0K
True
(2, 2, 6): 46K -> +32L -> +13K -> +1L -> +0K
True
(1, 3, 6): 28K -> +21L -> +10K -> +1L -> +0K
True
(5, 1, 4): 111K -> +127L -> +66K -> +14L -> +2K
True
(1, 5, 4): 86K -> +98L -> +48K -> +18L -> +9K
True
(2, 1, 6): 36K -> +33L -> +16K -> +2L -> +1K
True
(1, 2, 6): 36K -> +33L -> +16K -> +2L -> +1K
True
(1, 1, 6): 16K -> +18L -> +12K -> +3L -> +0K
True
(6, 1, 6): 3K -> +1L -> +0K
True
(5, 2, 6): 3K -> +1L -> +0K
True
(4, 3, 6): 3K -> +1L -> +0K
True
(3, 4, 6): 3K -> +1L -> +0K
True
(2, 5, 6): 3K -> +1L -> +0K
True
(1, 6, 6): 3K -> +1L -> +0K
True
(6, 1, 5): 9K -> +5L -> +2K
True
(5, 2, 5): 22K -> +14L -> +8K
True
(4, 3, 5): 32K -> +18L -> +8K
True
(3, 4, 5): 32K -> +18L -> +8K
True
(2, 5, 5): 24K -> +15L -> +8K
True
(1, 6, 5): 7K -> +4L -> +2K
True
(6, 1, 4): 15K -> +10L -> +4K
True
(1, 6, 4): 16K -> +10L -> +3K
True
(3, 1, 3): 32K -> +116L -> +161K -> +422L -> +272K
True
(1, 3, 3): 78K -> +277L -> +225K -> +442L -> +240K
True
(2, 1, 4): 21K -> +65L -> +72K -> +186L -> +137K
True
(1, 2, 4): 42K -> +139L -> +113K -> +219L -> +128K
True
(4, 1, 2): 21K -> +65L -> +72K -> +186L -> +137K
True
(1, 4, 2): 49K -> +164L -> +125K -> +226L -> +132K
True
(6, 2, 5): 3K -> +1L -> +0K
True
(5, 3, 5): 3K -> +1L -> +0K
True
(4, 4, 5): 3K -> +1L -> +0K
True
(3, 5, 5): 3K -> +1L -> +0K
True
(2, 6, 5): 3K -> +1L -> +0K
True
(6, 1, 1): 16K -> +18L -> +12K -> +3L -> +0K
True
(1, 6, 1): 12K -> +18L -> +15K
True
(1, 1, 5): 6K -> +16L -> +16K
True
(5, 1, 1): 6K -> +16L -> +16K
True
(1, 5, 1): 10K -> +28L -> +20K
True
(6, 6, 1): 3K -> +1L -> +0K
True
(5, 6, 1): 7K -> +4L -> +2K
True
(6, 5, 1): 9K -> +5L -> +2K
True
\end{verbatim}
}

Summarizing, the Kontsevich graph encoding of the reduced affine star\/-\/product $\star_{\text{aff}}^{\text{red}}\text{ mod }\bar{o}(\hbar^7)$ is given 
in~\cite[Appendix~A.2]{affine7}
; its analytic formula is available 
from~\cite[Appendix~B]{affine7}
, and in this
~\ref{AppStarAffineReducedAssoc7} we have certified its associativity up to $\bar{o}(\hbar^7)$.

\end{document}